\documentclass[a4paper,12pt]{article}
\usepackage{verbatim}
\usepackage{amssymb}
\usepackage{amsbsy}
\usepackage{amscd}
\usepackage{amsmath}
\usepackage{amsthm}
\usepackage[mathscr]{eucal}
\usepackage{mathpazo}

\theoremstyle{definition}
\newtheorem{axiom}{Axiom}\numberwithin{axiom}{section}
\newtheorem*{axiom*}{Axiom}
\newtheorem{defn}[axiom]{Definition}
\newtheorem*{defn*}{Definition}

\newtheorem*{example*}{Example}

\newtheorem*{conjecture*}{Conjecture}

\newtheorem*{question*}{Question}

\newtheorem*{problem*}{Problem}
\theoremstyle{remark}
\newtheorem{remark}[axiom]{Remark}
\newtheorem*{remark*}{Remark}
\theoremstyle{plain}
\newtheorem{theorem}[axiom]{Theorem}
\newtheorem*{theorem*}{Theorem}
\newtheorem{lemma}[axiom]{Lemma}
\newtheorem*{lemma*}{Lemma}
\newtheorem{proposition}[axiom]{Proposition}
\newtheorem*{proposition*}{Proposition}
\newtheorem{corollary}[axiom]{Corollary}
\newtheorem*{corollary*}{Corollary}

\theoremstyle{definition}
\newtheorem{noname}[axiom]{\mbox{ }}
\theoremstyle{plain}
\newtheorem{claim}[axiom]{Claim}

\newcommand{\R}{{\mathbb R}}
\newcommand{\N}{{\mathbb N}}
\newcommand{\C}{{\mathbb C}}
\newcommand{\Z}{{\mathbb Z}}

\newcommand{\F}{{\mathcal F}}
\newcommand{\cG}{{\mathcal G}}
\newcommand{\cH}{{\mathcal H}}

\newcommand{\constdef}{$K\cup(\R^n\setminus (-2,2)^n)$}
\newcommand{\frp}{ foliated $\R^{n-2}$-product }
\newcommand{\frps}{ foliated $\R^{n-2}$-products }
\newcommand{\DR}{ \mathrm{Diff}_c\R^{n-2}}
\newcommand{\DRk}{ \mathrm{Diff}_c\R^{k}}
\newcommand{\CRk}{C^\infty_c(\R^k,\R^k)}
\newcommand{\cW}{\mathcal{W}}

\begin{document}

\begin{center}

{\Large 
{Thurston's h-principle for 2-dimensional 
Foliations \\
of Codimension Greater than One}}
\vspace{10pt}
\\
Yoshihiko MITSUMATSU\footnote
{This research is supported in part by the 
Grant-in-Aid for Scientific Research 
(No.\ 18340020), the Ministry of Education, Culture, 
Sports, Science and Technology, Japan.  }
and 
Elmar VOGT \footnote
{An essential portion of the research for this paper was carried out 
during two stays of the second author 
 in Tokyo, the first in 2006-2007 as a visiting professor at the  University of Tokyo,
the second in 2013 at Chuo University.  
\\
{2010 {\it Mathematics Subject Classification}. 
Primary \, 57R30, 57R32; Secondary \, 58D05,
\\
}
}
\vspace{10pt}
\\
{\it  Dedicated to Steven Hurder and Takashi Tsuboi 
on the occasion of their 60th birthdays}
\end{center}



\begin{abstract}
We recreate an unpublished proof of William Thurston from the early 1970's that any smooth $2$-plane field on a manifold of dimension at least $4$ is homotopic to the tangent plane field of a foliation.
  
\end{abstract}



\section{Introduction}\label{sect:intro}
The main purpose of this article is to write down in reasonable detail a proof of a theorem of Thurston which many experts know of and of whose existence many more are aware. The proof is due to Thurston but was never written up since it was superseded by a to some extent more elementary proof of a substantially more general theorem. Thus the theorem below became a mere corollary in Thurston's published work: Corollary 3 in \cite{Thur}. Here it becomes the main result we want to prove.

\begin{theorem}[Thurston]\label{thm:main} Any $C^\infty$ $2$-plane field $\tau$ on a manifold $M$ of dimension at least 4 is homotopic to an integrable one. If $\tau$ is already integrable in a neighborhood of a closed set $K\subset M$, then the homotopy can be chosen to be constant in a neighborhood of $K$.
\end{theorem}

The proof of Theorem \ref{thm:main} which we give in this paper we learned from Takashi Tsuboi who in turn had learned it from Andr\'e Haefliger, and Haefliger attributes it to William Thurston. Thurston remarks in \cite{Thur} that a ``proof of Corollary 3 by itself was the starting point leading to'' his main theorem in \cite{Thur}.

\begin{remark} Apart from salvaging some of  Thurston's ideas there is one further aspect which makes it worthwhile  to make this proof accessible. We are able to work entirely with plane fields on the given manifold. This is special to $2$-plane fields. With today's knowledge about $\mathrm{Diff}_c(\R^{n-k})$ our proof below shows that for $k\geq 2$ any smooth $k$-plane field on an $n$-manifold can be homotoped into one which is integrable in a neighborhood of the $(n-k+2)$-skeleton of some very fine triangulation. Thurston's proof in \cite{Thur} deals with plane fields on the product of the total space of the normal bundle of the original plane field with the unit interval. This pushes the dimension of the manifold on which one works up quite a bit, especially for 2-plane fields. 

For $2$-plane fields on $4$- and $5$- manifolds, or more generally on $n$-manifolds for those $n$ where the identity component of $\mathrm{Diff}_c(\R^{n-2})$ is simply connected, the foliations obtained after deforming the plane field have a particularly transparent description, as we will see in section \ref{sect:hole1}. 
\end{remark}

Theorem \ref{thm:main} is an easy consequence of 
\begin{proposition}\label{prop:main}
Let $\tau$ be a $C^\infty$ $2$-plane field of $\R^n$ which is integrable in a neighborhood of a closed subset $K$ of $\R^n$. Then  $\tau$ is homotopic to a $C^\infty$ $2$-plane field which is integrable in a neighborhood of $K\cup [-1,1]^n$. The homotopy can be chosen to be constant in a neighborhood of $K\cup (\R^n\setminus (-2,2)^n)$.
\end{proposition}
\begin{proof}[Proof of \textup{Theorem \ref{thm:main}}]
Choose a countable atlas $\{h_i:U_i\to \R^n\}_{i\in \N}$ for $M$ such that $\{h_i^{-1}((-1,1)^n)\}_{i\in \N}$ is a covering of $M$. Start with $K':= h_1(U_1\cap K)\subset \R^n$ and $\tau':=h_{1*}(\tau_{| U})$ as input for Proposition \ref{prop:main}. The output is a plane field $\tau_1'$ whose pull back to $U_1$ can be extended by $\tau$ to a plane field $\tau_1$ which is integrable in a neighborhood of $K\cup h_1^{-1}[-1,1]^n$ by  a homotopy which is constant in a neighborhood of $K\cup h_1^{-1}(\R^n\setminus (-2,2)^n)=:K_1$.

Continue by replacing $K$, $\tau$, and $h_1$ by $K_1$,$\tau_1$, and $h_2$. If $x\in \bigcup\limits_{i=1}^N h_i^{-1}((-1,1)^n)$ then there is a neighborhood $V$ of $x$ such that all homotopies after step $N$ are constant on $V$.  If furthermore $x$  is a point of $K$ then, if $V$ is small enough, also the homotopies of steps $1$ to $N$ are constant on $V$. Thus the sequence of plane fields $\tau, \tau_1, \tau_2,\ldots $ converges to a plane field having the desired properties.
\hfill \vspace{5pt} 
\end{proof}

The proof of Proposition \ref{prop:main} proceeds in three steps. These steps also appear in Thurston's paper \cite{Thur}. They are: 
\begin{itemize}
\item[(i)]Find a triangulation of $\R^n$ which is in general position with regard to $\tau$ in a neighborhood of $[-1,1]^n$
\end{itemize}
We recall Thurston's definition of general position in the following section. Step (i) corresponds to the Jiggling Lemma in \cite{Thur}.
\begin{itemize}
\item[(ii)] Deform the plane field $\tau$ into one which is civilized near $[-1,1]^n$ with respect to the triangulation of step (i).
\end{itemize}
This will be done in Section \ref{sect:civilization}. Roughly a plane field is civilized near some compact set $C$ in $\R^n$ with respect to a triangulation of some compact neighborhood of $C$ if the plane field is constant  along and tangential to the fibres of tubular neighbourhoods of the simplices satisfying some compatibility condition (see conditions \ref{normbdle1}, \ref{constonfibre1}, and \ref{compatible1} in Section \ref{sect:civilization}).

Civilization in Thurston's paper is  important to ensure that homotopies of plane fields were performed in the space of plane fields which are in general position. For us it is important for understanding  the plane field near the $(n-1)$-skeleton when we deal with the third and final step:
\begin{itemize}
\item[(iii)] Filling the holes.
\end{itemize}
A civilized $2$-plane  field $\tau$ is integrable, i. e. a foliation, in a neighborhood of the $(n-1)$-skeleton. Filling the holes  is the extension of this foliation to the interiors of the $n$-simplices, the holes, in such a way, that the resulting tangent plane field is homotopic to $\tau$, where as always the homotopy is constant in a neighborhood of $K\cup (\R^n\setminus (-2,2)^n)$. 

This is the most difficult step. The corresponding step in \cite{Thur}, Section 4, is easier since the use of a collapsible triangulation on the product of the total space $E$ of the normal bundle of $\tau$ with the unit interval $[0,1]$ allows  enough control over the boundary of the hole to make this step quite simple.  The reverse process of a simplicial collapse, called inflation in \cite{Thur}, starts at $E\times \{0\}$ and one needs a foliation there to start with. This foliation is provided by the hypothesis of the main theorem in \cite{Thur} that the normal bundle of $\tau$ is the normal bundle of a Haefliger structure, a necessary hypothesis, if one wants to be able to homotope $\tau$ to a foliation. For $2$-plane bundles this hypothesis is satisfied by results of Thurston, Mather, and Epstein \cite{ThurBull}, \cite{Ma:integr}, \cite{Ma:comm}, \cite{Ma:Cime}, \cite{Ma:curious}, \cite{Ep:smooth}.

When working directly on the manifold, as we do here, the fact that our plane field is civilized gives us after moving a little to the inside of  the $n$-simplices  a product structure $D^2\times D^{n-2}$ for  each  $n$-simplex such that the plane field is transverse to the $D^{n-2}$ factor and  parallel to the $D^2$ factor near $D^2\times \partial D^{n-2}$. Thus the restriction to $\partial D^2\times D^{n-2}$ is a foliated $D^{n-2}$-product over $S^1=\partial D^2$ with support in the interior of $D^{n-2}$.  We have to show that we can extend any such foliated product to a $2$-dimensional foliation on all of $D^2\times D^{n-2}$. Furthermore we want the plane field of this foliation to be  homotopic to one which is transverse to the $D^{n-2}$-factor by a homotopy which is constant near the boundary. Since $\tau$ is transverse to the $D^{n-2}$ factor, the plane field of the foliation will be homotopic to it by a homotopy which is constant near the boundary of the simplex.

Foliated $D^{n-2}$-products with support in the interior of $D^{n-2}$ are the same as foliated $\R^{n-2}$-products with compact support. These correspond up to fitting beginnings and ends to smooth paths in $\mathrm{Diff}_c\R^{n-2}$ beginning at the identity. Thurston in \cite{Thur}, Section 4, provides for $n\geq 4$ an explicit and easy to understand  filling for a non-trivial path.  
It is fairly easy to see that a path homotopic to one which can be filled can itself be filled. Also the concatenation of two paths can be filled if the individual paths could be. Obviously, the conjugate of a fillable path by an element of $\mathrm{Diff}_c\R^{n-2}$ is also fillable.    Assume now for simplicity that the identity path component of $\mathrm{Diff}_c\R^{n-2}$ is simply connected. Then homotopy classes of paths correspond bijectively to their end points. Since by results of Epstein \cite{Ep:simple} and of Epstein, Mather and Thurston cited above the identity path component of $\mathrm{Diff}_c\R^{n-2}$ is simple, we are done. In fact, using Theorem 2.2 and Lemma 3.1 of \cite{Tsu:unifsimp} we see that any path is homotopic to the concatenation of $8$ conjugates of Thurston's path in \cite{Thur}, Section 4. Thus one can practically see how each hole is filled.

The situation where the identity path component of $\mathrm{Diff}_c\R^{n-2}$ is not simply connected is more complicated. It is here where, as far as we know, a directly applicable theorem with a reasonably detailed published proof is unavailable. In a note at the bottom of page 226 in \cite{Thur} Thurston formulates  a theorem of him which would allow us to proceed along the lines described in the previous paragraph. But we have not found any place in the literature where this theorem is stated as a theorem. 

Nevertheless, when looking at the experts' papers one finds enough results and arguments to piece together a result that satisfies our needs. In particular,  the arguments used in the very last paragraph in \cite{Ep:smooth}, where Epstein shows that the universal covering group of $\mathrm{Diff}^\infty_c(M)_0$ is perfect, together with a good understanding of what is done before in \cite{Ep:smooth}, and also using Tsuboi's lemma in \cite{Tsu:unifsimp},  one can achieve our goal. 

We will go a slightly different way. Instead of \cite{Ep:smooth} we will use Proposition 2 in \cite{HalRyb}. This result is (reasonably) straightforward to state and  makes it easier to describe the necessary estimates for our arguments. The only drawback (for some it might be an advantage) is  that the authors of \cite{HalRyb} use a topology on $\mathrm{Diff}_c\R^{n-2}$ developed in   \cite{KriegMich} to do analysis on infinite dimensional manifolds modelled on locally compact vector spaces, which might be unfamiliar to some. But what we actually need is very little and will be explained in Subsection \ref{subsec:smoothperf}. We will again make use of Proposition 3.1 of \cite{Tsu:unifsimp}. In fact it is important to see its proof which we, in Subsection \ref{subsec:perfecttosimple}, will more or less verbally copy directly from \cite{Tsu:unifsimp}.

\begin{remark}
Corollary 3 in \cite{Thur}  (our Theorem \ref{thm:main}) follows from the main result there using  the fact that  the classifying space $B\overline{\Gamma}_k$ of  codimension $k$ Haefliger structures with trivialized normal bundle  is $(k+1)$-connected. This is a consequence of two theorems. The Thurston-Mather Theorem, which establishes an isomorphism   between the homology of the $k$-fold loop space of $B\overline{\Gamma}_k$ and the homology of the classifying space of the homotopy fibre $\overline{\mathrm{Diff}}_c\R^k$ of $\mathrm{Diff}_c^\delta\R^{k} \to \mathrm{Diff}_c\R^{k}$ where  $\mathrm{Diff}_c^\delta\R^{k} $  is $\mathrm{Diff}_c\R^{k}$ made discrete. The second theorem, for $0<r<\infty$, $r\not=k+1$, due to Mather  \cite{Ma:comm}, see also the appendix of \cite{Tsu:conn} for a somewhat different proof,  and for $r=\infty$  due to Thurston \cite{ThurBull}, states that the universal cover of the identity path component of $\mathrm{Diff}^r_c\R^{k}$ is perfect. Thurston's proof is outlined in \S 3 of \cite{Ma:curious}. A detailed proof for the case $r=\infty$ along the lines developed by Mather for $r>k+1$  in \cite{Ma:comm} can be found in the paper \cite{Ep:smooth} by David Epstein mentioned above. 

In the final section we will go the opposite  way: following Haefliger's proof in \cite{Hae:ouvert} (see also \cite{Hae:integrability}) that $\pi_k(B\overline{\Gamma}_k)=0$ we will use our main theorem to show that $\pi_{k+1}(B\overline{\Gamma}_k)=0$ for $k\geq2$. As in \cite{Hae:ouvert} this proof involves the Gromov-Phillips transversality theorem \cite{Grom}, \cite{Phil}. But note that  our proof of Proposition \ref{prop:main}   uses  the second of the two theorems mentioned above, the one about the perfectness of the universal cover of the identity component of  $\mathrm{Diff}^r_c\R^{k}$, in the version proved by Haller, Rybicki, and Teichmann in \cite{HalRyb}.
\end{remark}

\section{Triangulations in general position with respect to a plane field}\label{sect:genposition}
Let $1\leq k\leq n-1$ and let $\tau$ be a $C^1$ $k$-plane field on $\R^n$. The following terminology is due to Thurston \cite{Thur}.
\begin{defn}\label{def:genpos} An $n$-simplex $\sigma$ of $\R^n$ is in {\it general position with respect to $\tau$} if for every $x\in \sigma$ the orthogonal projection of $\R^n$ onto the $(n-k)$-plane orthogonal to $\tau(x)$ is injective on every $(n-k)$-face of $\sigma$. A triangulation of $\R^n$ by affine simplices is in {\it general position with respect to $\tau$ in a neighborhood of a closed subset $K\in\R^n$} if every $n$-simplex intersecting $K$ is in general position with respect to $\tau$.
\end{defn}
\begin{defn}\label{def:epsjiggle} Let $T$ be an affine triangulation of $\R^n$ and $\varepsilon>0$. An $\varepsilon$-jiggling of $T$ is an affine triangulation $T'$ of $\R^n$ such that there is a simplicial isomorphism $\phi: T\to T'$ with $\|\phi(v)-v\|<\epsilon$ for every vertex $v\in T$.
\end{defn}
In section 6 of \cite{Thur} the following statement is proved:
\begin{proposition} \label{prop:jiggle}Let $\tau$ be a $C^1$ $k$-plane field of $\R^n$, $1\leq k\leq n-1$, let $K\subset\R^n$ be compact, and let $\varepsilon>0$. Then there exist an $L\in \N$ and for every $l\geq L$ an $\varepsilon$-jiggling of the standard triangulation associated to the cubical lattice $(\frac{1}{l}\Z)^n\subset \R^n$ which is in general position with regard to $\tau$ in a neighborhood of $K$.
\end{proposition}
Recall that the standard triangulation of the unit cube $I^n=[0,1]^n$ of $\R^n$ has the vertices of $I^n$ as vertices and for each permutation $\sigma\in S_n$ an $n$-simplex $\langle v_0,\ldots, v_n\rangle$ with $v_0=(0,\ldots, 0)$ and $v_i$ being obtained from $v_{i-1}$ by replacing the $0$ in the $\sigma(i)$-th coordinate by a $1$. Triangulating every cube $x+I^n$, $x\in \Z^n$ by translating the triangulation of $I^n$ by $x$ gives a triangulation of $\R^n$. The standard triangulation associated to the lattice  $(\frac{1}{l}\Z)^n$ is obtained from this by multiplying every vertex by $1/l$.

Thurston proves a more general result, but for us Proposition \ref{prop:jiggle} suffices. His proof is concise, clear, and to the point. There are a couple of misprints, but of the ``self correcting'' type. So there is no reason to repeat it here.

We will apply Proposition \ref{prop:jiggle} in our proof of Proposition   \ref{prop:main} for the given plane field $\tau$ and $K=[-2,2]^n$.

\section{Civilization}\label{sect:civilization}
Given $K\subset \R^n$ closed and $\tau$ a smooth $2$-plane field which is integrable in a neighborhood $U$ of $K$, let $T$ be an $\varepsilon$-jiggling of the standard triangulation associated to the lattice $(\frac{1}{l}\Z)^n$ which is in general position with respect to $\tau$ in a neighborhood of $[-2,2]^n$. Choose $l$ large enough so that the following holds:
\begin{noname}\label{farfromK}
 If $x\in U\cap [-2,2]^n$ is a vertex of $T$ whose closed star $\overline{st}(x,T)$  is not contained in $U$ then $\overline{st}(x,T)\cap K = \emptyset$
\end{noname}
\begin{noname}\label{insidecube} If $x$ is a vertex of $T$ with $\overline{st}(x,T)\cap [-1,1]^n\neq\emptyset$ then $\overline{st}(x,T)\subset [-3/2,3/2]^n$.
\end{noname}
\begin{noname} \label{noname:graphprop} If $\sigma$ is an $n$-simplex of $T$ which intersects $[-2,2]^n$ then for any $x,y\in\sigma$ the $2$-plane $\tau(y)$ is the graph of a linear map $L_{x,y}:\tau(x)\to \tau(x)^\bot$ of norm  less than $1$.
\end{noname}
We denote by $T_1$ the union of all simplices  which are faces of $n$-simplices of $T$ which intersect $[-1,1]^n$.

In this section we deform the plane field $\tau$ into a plane field which is integrable in a neighborhood of the union of $K$ and the $(n-1)$-skeleton of $T_1$. The deformation is done in the space of smooth plane fields for which $T$ is in general position near $[-2,2]^n$, and which also satisfy \ref{noname:graphprop}. Furthermore, the deformation is constant in a neighborhood of \constdef.

The idea is to deform $\tau$ in a tubular  neighborhood of each simplex of $T_1$, starting with the $0$-simplices and extending it dimension by dimension up to the $(n-1)$-skeleton of $T_1$. To ensure that near $K$ the deformation is constant only simplices which are faces of an $n$-simplex not contained in $U$ trigger a deformation. Notice that a face of a simplex $\sigma$ may be the face of an $n$-simplex not contained in $U$ while $\sigma$ is not a face of such an $n$-simplex. Thus deformations of $\tau$ may take place in tubular neighborhoods of the boundary of a simplex $\sigma$, even if $\sigma$ itself does not trigger a deformation. 

The deformation process is inductive  by induction on the dimension $p$ of the skeleta of $T_1$ up to dimension $n-2$. The extension of the deformation from a neighborhood of the $(n-2)$-skeleton to a neighborhood of the $(n-1)$-skeleton differs mildly from the previous induction steps in as far as the plane field at a point of an $(n-1)$-simplex intersects the tangent plane of the simplex non-trivially. 

So let $0\leq j\leq n-2$.  We say that a smooth $2$-plane field  $\tau_j$ is civilized on the $j$-skeleton of $T_1$ if it is a $2$-plane field with respect to which $T$ is in general position near $[-2,2]^n$, and there are real numbers  $\delta_0 >\ldots >\delta_j>0$ and $\eta_0 >\ldots > \eta_j>0$   having the  properties \ref{normbdle}-\ref{integrableinU} below.
\begin{noname}\label{normbdle}
Denote for any point $x$ of an $i$-simplex $\sigma$ of $T_1$ with $0\leq i \leq j$ the closed $\delta$-neighborhood of $x$ in the affine $2$-plane $x+\tau_j(x)$ by $B_x(\delta)$ and the closed $\eta$-neighborhood of $x$ in $x+E_x$ by $E_x(\eta)$, where $E_x$ is the orthogonal complement  of $\tau_j(x) + T_\sigma$ with $T_\sigma$ the tangent plane of $\sigma$. Then, if $\sigma$ is a face of an $n$-simplex not contained in $U$,  the $(n-i)$-disks $B_x(\delta_i)\times E_x(\eta_i)$, $x\in \sigma$, are the fibres of a tubular neighborhood $N(\sigma)$ of $\sigma$ in $\R^n$, and any $(n-2)$-simplex of $T_1$ having $\sigma$ as a face intersects $\partial(B_x(\delta_i)\times E_x(\eta_i))$ in a subset of $\mathring{B}_x(\delta_i)\times \partial E_x(\eta_i)$.
\end{noname}
\begin{noname}\label{constonfibre}
For any $0\leq i\leq j$ and any point $x$ in an $i$-simplex of $T_1$ in the boundary of an $n$-simplex not contained in $U$  the plane field  $\tau_j$ is equal to $\tau_j(x)$ on $B_x(\delta_i)\times E_x(\eta_i)$. In particular, the plane field $\tau_j$ is integrable on $N(\sigma)$ and tangent to the fibres.
\end{noname}
\begin{noname}\label{compatible}
For any two simplices $\sigma$, $\sigma'$ of $T_1$ of dimension at most $j$ where each, $\sigma$ and $\sigma'$, is in the boundary of an $n$-simplex not contained in $U$ we have $N(\sigma)\cap N(\sigma')\subset N(\sigma\cap\sigma')$. Furthermore, if $\sigma'$ is a proper face of $\sigma$ and the fibre $B_x(\delta_i)\times E_x(\eta_i)$ of $N(\sigma)$ intersects $N(\sigma')$, say $(y_1,y_2)\in  B_x(\delta_i)\times E_x(\eta_i)$ lies in $B_v(\delta_{i'})\times \{y_2' \}$ with $v\in \sigma'$ and $y_2'\in E_v(\eta_{i'})$, $i'=\textrm{dim}\sigma'$, then  $B_x(\delta_i)\times\{y_2\} \subset \mathring{B}_v(\delta_{i'})\times \{y_2'\}$. In addition we demand that $N(\sigma)\cap \sigma'' = \emptyset$ for any simplex $\sigma''$ of dimension at least $j+1$ of which $\sigma$ is not a face.
\end{noname}
\begin{noname}\label{tubeinU}
If $0\leq i\leq j$ and the $i$-simplex  $\sigma$ is contained in $U$ and in the boundary of an $n$-simplex not in $U$, then $N(\sigma)\subset U$.
\end{noname}
\begin{noname}\label{integrableinU} $\tau_j$ is integrable in an open subset of $U$ containing 
$$ K,$$
$$\cup \{N(\sigma) \vert \sigma\in T_1, \textrm{dim}\sigma \leq j, \sigma\subset U\textrm{ and } \sigma\leq\rho\textrm{ with } \rho\not\subset U\},
$$
$$ \textrm{ and } \cup\{\sigma\vert \sigma\in T_1\textrm{ and } \sigma\subset U\}$$.
\end{noname}

By abuse of notation, we call any smooth $2$-plane field with respect to which the triangulation  $T$ of Section \ref{sect:genposition} is in general position near $[-2,2]^n$ civilized on the $(-1)$-skeleton of $T_1$. Set $\delta_{-1}$ and $\eta_{-1}$ to be $\infty$.

We then make the following 
\begin{claim}\label{claim:civilinduction} Let $-1\leq p\leq n-2$,and let $\tau_{p-1}$  be a $2$-plane field which is civilized on the $(p-1)$- skeleton of $T_1$, and let $\delta_0 >\ldots >\delta_{p-1}>0$ and $\eta_0 >\ldots > \eta_{p-1}>0$ be the associated real numbers. Then there exist a smooth $2$-plane field $\tau_p$, $0<\delta_p<\delta_{p-1}$, and $0<\eta_p<\eta_{p-1}$ such that properties \ref{normbdle}-\ref{integrableinU} hold for $j=p$ and $\delta_i$, $\eta_i$, $i=0,\ldots,p$. Furthermore, $\tau_p$ is homotopic to $\tau_{p-1}$ in the space of smooth $2$-plane fields with respect to which $T$ is in general position near $[-2,2]^n$ and which satisfy  \ref{noname:graphprop} . Furthermore, the homotopy  is constant in a neighborhood of \constdef.
\end{claim}
\begin{proof}[Proof of \textup{Claim \ref{claim:civilinduction}}] The idea is very simple. Look at all $p$-simplices $\sigma$ of $T_1$ which are a face of an $n$-simplex not contained in $U$. For all proper faces $\sigma'$ of such a simplex the plane field $\tau_{p-1}$ is constant along the fibres $B_y(\delta_i)\times E_y(\eta_i)$, $i= \textrm{dim }\sigma'$, and $\sigma$ intersects these fibres in $\mathring{B}_y(\delta_i)\times E_y(\eta_i)$ (see \ref{normbdle}). Thus, if $x$ is a point of $\sigma$ lying in the fibre  $B_y(\delta_i)\times E_y(\eta_i)$ of $N(\sigma')$, then $\tau_{p-1}(x)=\tau_{p-1}(y)$, and the orthogonal complement $E_x$ of $T_x\sigma + \tau_{p-1}(x)$ is a subspace of $E_y$. Therefore, we find preliminary $\delta_p$ and $\eta_p$  not depending on $\sigma$ such that \ref{normbdle}, \ref{compatible}, and \ref{tubeinU} hold. The last condition in \ref{normbdle} will be satisfied if $\eta_p / \delta_p$ is sufficiently small. Also \ref{constonfibre} holds for those $x\in\sigma$ which are for some proper face $\sigma'$ of $\sigma$ in $N(\sigma')$, say $x= (u,v)\in B_y(\delta_i)\times E_y(\eta_i)$, with $E_x(\eta_p)\subset \{u\}\times E_y(\eta_i)$.

We now want to deform $\tau_{p-1}$ keeping \ref{integrableinU} in mind so that \ref{constonfibre} holds for every $x\in \sigma$. The deformation will have support in the complement of the tubular neighborhoods $N(\sigma')$ where the $\sigma'$ are the proper faces of $\sigma$.

We distinguish two cases:\\[1mm]
{\bf Case 1.} $\sigma \not\subset U$.\\[1mm]
By the last condition in \ref{compatible} (which is satisfied including the preliminary tubular neighborhoods of the $p$-simplices) and by  \ref{noname:graphprop}, for every $x\in\sigma$ and $z\in B_x(\delta_p)\times E_x(\eta_p)$ the plane $\tau_{p-1}(z)$ is the graph of a linear map $f_{zx}:\tau_{p-1}(x)\to \tau_{p-1}(x)^\perp$ of norm less than $1$. Then $\tau_{p-1,t}(z)$ is defined to be the graph of $(1-t)f_{zx}$. There are $\overline{\delta}_p>\delta_p$ and $ \overline{\eta}_p>\eta_p$ such that the $B_x(\overline{\delta}_p)\times E_x(\overline{\eta}_p)$, $x\in\sigma$, still are the fibres of a tubular neighborhood $\overline{N}(\sigma)$  of $\sigma$. Use these larger fibres to slow down the homotopy of $\tau_{p-1}$ when moving ``radially'' from $\partial(B_x(\delta_p)\times E_x(\eta_p))$ to $\partial(B_x(\overline{\delta}_p)\times E_x(\overline{\eta}_p))$ so that the homotopy is supported in $\overline{N}(\sigma)$.

Note that during the homotopy condition \ref{noname:graphprop} is always satisfied: if $y$ is a point of an $n$-simplex $\rho$ with $z\in \rho$ then also $x$ is in $\rho$ and there are linear maps $f_{xy},f_{zy}:\tau_{p-1}(y)\to \tau_{p-1}(y)^\perp$ such that $\tau_{p-1}(x)$ is the graph of  $f_{xy}$ and $\tau_{p-1}(z)$ is the graph of $f_{zy}$. Then $\tau_{p-1,t}(z)$ is the graph of $(1-t)f_{zy}+tf_{xy}$ which has norm at most equal to the maximum of the norms of $f_{xy}$ and $f_{zy}$. Also the homotopy is constant in every $N(\sigma')$ with $\sigma'$ a proper face of $\sigma$.\\[1mm]
{\bf Case 2.} $\sigma \subset U$.\\[1mm]
Notice first that applying the same homotopy as above the following may happen. In the annulus bundle $\overline{N}(\sigma)\setminus \mathring{N}(\sigma)$ over $\sigma$, where we slow down the homotopy, the plane field may become non-integrable although it was integrable there beforehand. Thus an $n$-simplex which was in the open set where $\tau_{p-1}$ is integrable may not be in the set where $\tau_{p}$ is integrable. But this would force us to reconsider all faces of this simplex and start the process all over. Eventually we would  have to deal with simplices which intersect $K$. This is the reason why we treat simplices contained in $U$ differently.

Since $\tau_{p-1}$ is integrable in a neighborhood of $N(\sigma)$ there is a tubular neighborhood $\varphi: E(\sigma)\times B(\delta_p)\to \R^n$ of $E(\sigma):= \bigcup\{\{0\}\times E_x(\eta_p)\vert x\in \sigma\}$ the fibres of which are $\delta_p$-neighborhoods of $z\in E(\sigma)$ in the leaf through $z$ of the foliation defined by $\tau_{p-1}$ in a neighborhood of $N(\sigma)$. By the uniqueness theorem for tubular neighborhoods there is an ambient isotopy of $\R^n$ which maps this tubular neighborhood to the tubular neighborhood $\psi: E(\sigma)\times B(\delta_p)\to \R^n$ which has $\{z\}\times B_x(\delta_p)$, $z\in E_x(\eta_p)$, $x\in \sigma$, as fibres. We need to straighten the foliation only in a small neighborhood of $\sigma$. So it suffices to isotope the image of $E(\sigma)\times B(\delta)$ of $\varphi$ into the image of $\psi$. The standard way of producing such an isotopy is by linearization of the fibres of $\varphi$ at points of $E(\sigma)$ and then connecting  the linear embeddings to the one of $\psi$ by the straight line segment in the space of linear embeddings (see e.~g. \cite{BroJa}, end of the proof of (12.13)). This first gives an isotopy between the (small) tubular neighborhoods, i.e. a level preserving embedding  $\Phi: [0,1]\times (E(\sigma)\times B(\delta))\to [0,1]\times \R^n$. This is extended to an ambient isotopy in the usual way  interpreting  the isotopy as a vectorfield on the image of $\Phi$ and extending this vectorfield to all of $[0,1]\times R^n$ by slowing it down to the zero vector field within a small neighborhood of the image of $\Phi$. Doing the slowing down process at level $t$ along the fibres of the tubular neighborhood given by the restriction of $\Phi$ to level $t$ will have the effect that the ambient isotopy will have support in $N(\sigma)$ and in the complement of the union of the interiors of the $N(\sigma')$ with $\sigma'$ a proper face of $\sigma$.

It is then clear that for a small enough $\delta=:\delta_p$ conditions \ref{noname:graphprop}-\ref{integrableinU} will be satisfied  with the possible exception of the last condition of \ref{normbdle} which concerns the intersections of the fibres of $N(\sigma)$ with $(n-2)$-simplices. But we noticed already that we can achieve this by shrinking $\eta_p$.
\hfill \vspace{5pt} 
\end{proof}

As the final step in this section we want to homotope $\tau_{n-2}$ to $\tau_{n-1}$ which is integrable in a neighborhood of the $(n-1)$-skeleton of $T_1$ and satisfying conditions \ref{noname:graphprop}-\ref{integrableinU} with \ref{normbdle}, \ref{constonfibre} and \ref{compatible} being replaced by their obvious analogues \ref{normbdle1}, \ref{constonfibre1} and \ref{compatible1}. 

\begin{noname}\label{normbdle1}
Denote for any point $x$ of an $(n-1)$-simplex $\sigma$ of $T_1$  the closed  $\delta$-neighborhood of $x$ in the affine line $x+F_x$ by $F_x(\delta)$, where $F_x$ is the orthogonal complement  in $\tau_{n-1}(x)$   of $ T_\sigma\cap \tau_{n-1}(x)$. Then, if $\sigma$ is a face of an $n$-simplex not contained in $U$,  the segments $ F_x(\delta_{n-1})$, $x\in \sigma$, are the fibres of a tubular neighborhood $N(\sigma)$ of $\sigma$ in $\R^n$.
\end{noname}
\begin{noname}\label{constonfibre1}
For any $0\leq i\leq n-2$ and any point $x$ in an $i$-simplex $\sigma$ of $T_1$ in the boundary of an $n$-simplex not contained in $U$  the plane field  $\tau_{n-1}$ is equal to $\tau_{n-1}(x)$ on $B_x(\delta_i)\times E_x(\eta_i)$. In particular, the plane field $\tau_{n-1}$ is integrable on $N(\sigma)$ and tangent to the fibres. Furthermore for any point $x$ in an $(n-1)$-simplex $\sigma$ of $T_1$ in the boundary of an $n$-simplex not contained in $U$  the plane field  $\tau_{n-1}$ is equal to $\tau_{n-1}(x)$ on $F_x(\delta_{n-1})$. In particular, the plane field $\tau_{n-1}$ is integrable on $N(\sigma)$ and  the fibres are tangent to the leaves, i.~e. the plane field in $N(\sigma)$ is the pullback of the integrable line field $T_\sigma \cap \tau_{n-1}(x)$, $x\in\sigma$, by the tubular neighborhood projection.
\end{noname}
\begin{noname}\label{compatible1}
\ref{compatible} holds for $p=n-1$. Furthermore for any two $(n-1)$-simplices $\sigma$, $\sigma'$ of $T_1$, each in the boundary of an $n$-simplex not contained in $U$, we have $N(\sigma)\cap N(\sigma')\subset N(\sigma\cap\sigma')$. We also demand that if $\sigma''$ is a proper face of $\sigma$ and the fibre $F_x(\delta_{n-1})$ of $N(\sigma)$ intersects $N(\sigma')$, say $y \in  F_x(\delta_{n-1})$ lies in $B_v(\delta_{i'})\times \{y_2' \}$ with $v\in \sigma'$ and $y'\in E_v(\eta_{i'})$, $i'=\textrm{dim}\sigma'$, then  $F_x(\delta_{n-1}) \subset \mathring{B}_v(\delta_{i'})\times \{y'\}$. In addition we demand that $N(\sigma)\cap \sigma''' = \emptyset$ for any $n$-simplex  $\sigma'''$ of which $\sigma$ is not a face.
\end{noname}

The proof that such a homotopy exists in the space of plane fields satisfying all our usual requirements and with support in the complement of the interiors of all $N(\sigma')$ with $\mathrm{dim}\sigma'<n-1$ is entirely analogous to the proof of Claim \ref{claim:civilinduction} and is omitted.

\begin{remark}\label{rem:civil} 
A plane field satisfying, like our plane field $\tau_{n-1}$, conditions \ref{normbdle1}-\ref{compatible1} is called civilized in \cite{Thur}. In the next two sections we will see that a civilized plane field which also satisfies \ref{noname:graphprop}, \ref{tubeinU}, and \ref{integrableinU} can be deformed into a plane field which is integrable in a neighborhood of $T_1$ by a homotopy which is constant in a neighborhood of \constdef , thus completing the proof of our main proposition, Proposition \ref{prop:main}.
\end{remark}

\section{Filling the hole. Part I}\label{sect:hole1}
We now may assume that the plane field $\tau$ of Proposition \ref{prop:main} satisfies conditions  \ref{normbdle1}-\ref{compatible1}  and also \ref{noname:graphprop}, \ref{tubeinU}, and \ref{integrableinU}.  In this section we complete the proof of Proposition \ref{prop:main} for the dimensions $n$ ($\geq 4$) for which the identity path component of  $\mathrm{Diff}_c\R^{n-2}$ is simply connected. This includes $n=4$ and $5$  (see Remark \ref{vanishhgroups} at the end of this section).

By \ref{integrableinU} we need to homotope $\tau$ only in a compact part of the interior of $n$-simplices $\sigma$ which are not contained in $U$.  For these $\sigma$ we have from \ref{normbdle1} - \ref{compatible1} explicit information about what $\tau$ looks like near their boundaries $\dot{\sigma}$, namely on $N(\dot{\sigma}):=\cup\{N(\sigma')\vert \sigma' \textrm{ a proper face of }\sigma\}$. Any subset of $\sigma$ diffeomorphic to $B^n$ or to $B^2\times B^{n-2}$ containing the complement of $\mathring{N}(\dot{\sigma})$ in $\sigma$ we call a hole. Our task is to fill  for each $\sigma$ some hole by deforming $\tau$ in $\sigma$ relative with respect to the complement of the hole  to an integrable field . Then by  \ref{farfromK} and \ref{insidecube} we are done.
\subsection{A product structure for the holes adapted to $\tau$.}
In this subsection we show that for each $n$-simplex $\sigma$ of $T_1$ not contained in $U$ there is a smooth embedding $\varphi: D^2\times  D^{n-2} \to \mathring{\sigma}$ such that the pullback $\tau':=\varphi^*\tau$ of $\tau$ has the properties \ref{projnearbdy}, \ref{constnearbdy}, and \ref{transversetofactor} below. Here $D^k$ is the closed unit disk in $\R^k$.
\begin{noname}\label{projnearbdy} Near $D^2\times\partial D^{n-2}$ the plane field $\tau'$ is the kernel of derivative of the projection to the $D^{n-2}$-factor.
\end{noname}
\begin{noname}\label{constnearbdy}
$\tau'$ is transverse to $\partial D^2\times D^{n-2}$ and is in a neighborhood of $\partial D^2\times D^{n-2}$ the pullback of the line field $\tau'\cap T(\partial D^2\times D^{n-2})$ by the projection $(D^2\setminus \{0\})\times D^{n-2}\to \partial D^2\times D^{n-2}$. Furthermore the line field $\tau'\cap T(\partial D^2\times D^{n-2})$ is transverse to $\{x\}\times D^{n-2}$ for all $x\in \partial D^2$.
\end{noname}
\begin{noname}\label{transversetofactor}
$\tau'$ is homotopic to a plane field which is transverse to $ \{x\}\times D^{n-2}$ for all $x\in D^2$.
\end{noname}

\begin{proof} Consider for some $x\in\sigma$ the orthogonal projection $p_x:\R^n\to A_x:=\tau(x)^\perp$. The image $p_x(\sigma)$ is a convex polytope $P^{n-2}\subset A_x\cong\R^{n-2}$. Because of general position $p_x$ restricted to  every $(n-2)$-face of $\sigma$ is an affine homeomorphism  onto its image. Therefore there is an $(n-3)$ dimensional subcomplex $\Sigma$ of $\sigma$ such that $p_x\vert_\Sigma:\Sigma \to \partial P^{n-2}$ is a simplicial isomorphism. The subcomplex $\Sigma$ is independent of the choice of $x\in\sigma$. In particular, $\tau$ is transverse to to the boundary $\dot{\sigma}$ of $\sigma$ in the complement of $\Sigma$. 

$p_x:\mathring{\sigma}\to\mathring{P}^{n-2}$ is a trivial fibre bundle with fibre an open $2$-disk. This provides us with a diffeomorphism $\varphi_1: \R^2\times \mathring{P}^{n-2}\to \mathring{\sigma}$. We may choose $\varphi_1$ so that $\varphi_1(S^1\times  \mathring{P}^{n-2})$ is close enough to $\dot{\sigma}$ to be in the civilized neighborhood of $\dot{\sigma}$ and such that $\tau$ is transverse to $\varphi_1(S^1\times  \mathring{P}^{n-2})$ and the line field $\tau\cap T(S^1\times \mathring{P}^{n-2})$ is transverse to every $\{x\}\times \mathring{P}^{n-2}$, $x\in S^1$. Here $S^1$ is the unit circle in $\R^2$. 

We identify $\mathring{P}^{n-2}$ with $\R^{n-2}$ in such a way that the inverse image of  complement of $D^{n-2}$ under $p_x$ is contained in the civilized neighborhood of $\Sigma$. Furthermore we can arrange for the zero section of the bundle, i. e. $\varphi_1(\{0\}\times \mathring{P}^{n-2})$, to be transverse to $\tau$ in the complement of an $(n-2)$-disk contained in the interior of $D^{n-2}$.

To achieve \ref{projnearbdy} we change $\varphi_1$ by a diffeotopy of $\mathring{\sigma}$ with support in a neighborhood of the image of $D^2\times S^{n-3}$. The diffeotopy is an ambient isotopy which moves the tubular neighborhood $D^2\times U$ of $\{0\}\times U$ to the tubular neighborhood of $\{0\}\times U$  the fibres of which are disks in the leaves of $\tau$  with center the corresponding point of $U$. Here $U$ is a small annular neighborhood of $S^{n-3}$. 

A second diffeotopy will give us \ref{constnearbdy}. From the fact that the image of $S^1\times \R^{n-2}$ is inside the civilized neighborhood of $\dot{\sigma}$ gives us a tubular neighborhood of $S^1\times D^{n-2}$ such that the plane field in this neighborhood is the pullback under the tubular neighborhood map of the line field induced on $S^1\times D^{n-2}$. If we are in $N(\rho)$ with $\textrm{dim}\rho=n-1$ and not in an $N(\rho')$ with $\rho'$ a proper face of $\rho$ this is part of the civilization structure. In neighborhoods of lower dimensional simplices this line field can be obtained by interpolation proceeding inductively by decreasing dimension. The ambient isotopy then moves the standard tubular neighborhood of $S^1\times D^{n-2}$ in $\R^2\times D^{n-2}$ into this tubular neighborhood.

Denote by $\varphi_2$ the diffeomorphism $\varphi_1$ after having subjected it to the two diffeotopies above. Then $\varphi$ is the restriction of $\varphi_2$  to $D^2\times D^{n-2}$.

Notice that \ref{transversetofactor} is clear since by \ref{noname:graphprop} the field $\tau$ is homotopic to the constant plane  field $\textrm{ker}dp_x$  given by the kernel of the differential of $p_x$. Thus $\varphi^*(\tau)$ is homotopic to $\varphi^*(\textrm{ker}dp_x)$, which is transverse to the factors $\{x\}\times D^{n-2}$, $x\in D^2$.
\end{proof}

\subsection{Foliated $\R^{n-2}$-products.}
Recall that for a $k$-manifold $M$ a foliated $M$-product over the manifold $X$ is a codimension $k$ foliation on $X\times M$ which is transverse to the second factor. The foliated $M$-product is said to have compact support if there is a compact subset $C$ of $M$ such that on $X\times (M\setminus C)$ the foliation is given by projection onto $M\setminus C$, i.~e. for any $y\in M\setminus C$ the connected components of $X\times \{y\}$ are leaves of the foliation. The last sentence of Property \ref{constnearbdy} says that $\tau'$ is a foliated $D^{n-2}$-product over $S^1=\partial D^2$, while Property \ref{projnearbdy} says that the restriction of $\tau'$ to $S^1\times \mathring{D}^{n-2}$ is a foliated $\mathring{D}^{n-2}$-product over $S^1$ with compact support. The first part of Property \ref{constnearbdy} states that for for some $\varepsilon >0$ for all $1-\varepsilon<r\leq1$ the restriction of $\tau'$ to $rS^1\times D^{n-2}$ is a foliated $D^{n-2}$-product which is independent of $r$ in the sense that the obvious diffeomorphism $rS^1\times D^{n-2}\to S^1\times D^{n-2}$ is foliation preserving. Since any foliated $\mathring{D}^{n-2}$-product over $S^1$ with compact support extends uniquely to a foliated $D^{n-2}$-product over $S^1$ which is given by projection to the second factor near $S^1\times S^{n-3}$ we focus our attention only these, i.~e. we are looking at foliated $\R^{n-2}$ bundles over $S^1$ with compact support. 

Thus, what remains to be done is to find for any foliated $\R^{n-2}$-product over $S^1$ with compact support a $2$-dimensional foliation on $D^2\times\R^{n-2}$ which satisfies:
\begin{noname}\label{extend1}
Outside some compact $D^2\times C$ it is  given by projection onto $\R^{n-2}$,
\end{noname}
 \begin{noname}\label{extend2} it is transverse to $rS^1\times \R^{n-2}$for $r$ close enough to $1$, and induces there the given foliated $\R^{n-2}$-bundle,
 \end{noname}
and
\begin{noname}\label{transversetofactfill}
the tangent plane field to the foliation is homotopic to one which is transverse to the $\R^{n-2}$-factor by a homotopy which is constant near $S^1\times \R^{n-2}$.
\end{noname}
To do this, we change our perspective. Notice that any  \frp $\xi$ over $S^1$ with compact support defines a path $w_\xi :\R\to \mathrm{Diff}_c\R^{n-2}$ starting at the identity which describes the leaf of $\xi$ through the point $(1,y)\in \{1\}\times \R^{n-2} $ as the set $\{(\exp(2\pi it), w_\xi(t)(y))\vert t\in \R\}$. The path $w_\xi$ is periodic in the sense that for each integer $m$ we have $w_\xi(m+t)=w_\xi(t)\circ w_\xi(m)$ so that $w_\xi$ is determined by its restriction to $[0,1]$. Furthermore there is a compact subset $C$ of  $\R^{n-2}$ such that for all $t\in \R$ the diffeomorphism $w_\xi(t)$ restricted to the complement of $C$ is the identity.

Conversely, we would like to associate to periodic curves $w:\R\to \mathrm{Diff}_c\R^{n-2}$ a \frp over $S^1$ with compact support.

At this point a few remarks about the topology and smoothness structure of $\DR$ are called for. For compact $K\subset\R^{n-2}$ denote by $\mathrm{Diff}_K\R^{n-2}$ the space of diffeomorphisms of $\R^{n-2}$ with support in $K$ with its usual compact open $C^\infty$-topology. Then $\DR$ as a set is the inductive limit of the $\mathrm{Diff}_K\R^{n-2}$ where $K$ ranges over the compact subsets of $\R^{n-2}$. Then any compact subset of $\DR$ is contained in some $\mathrm{Diff}_K\R^{n-2}$. In particular the image of any continuous map from $\R$ into $\DR$ restricted to a bounded subset is contained in some $\mathrm{Diff}_K\R^{n-2}$.

A smooth curve of $\DR$ is a map $w:\R\to \DR $ having the following two properties:
\begin{noname} \label{smoothadj}The adjoint map $w^\vee:\R\times\R^{n-2}\to\R^{n-2}$ given by $w^\vee(t,x)=w(t)(x)$ is smooth.
\end{noname}
\begin{noname}\label{compactsupp}
For each bounded subset $J$ of $\R$ there is a compact $K\subset\R^{n-2}$ such that for every $x\not\in K$ the path $t\mapsto w(t)(x)$ is constant on $J$. (Compare 30.9 and 42.5 in \cite{KriegMich}; be not deterred by the use of the $c^\infty$-topology in their definition of infinite dimensional manifolds since this does not change the set of smooth curves.)
\end{noname} 

Thus periodic smooth curves of $\DR$ correspond bijectively to smooth \frps over $S^1$ with compact support. A map $w:[0,1]\to \DR$   is called a smooth periodic curve if   $w(0)= \mathrm{id}$ and its extension $w:\R\to \DR$ given by $w(t+m):= w(t)\circ w(1)^m$, $t\in [0,1]$, $m\in \Z$,  is smooth. The extension is then obviously periodic. We will call a smooth periodic curve $w:[0,1]\to \DR$ {\it fillable} if the associated \frp over $S^1$ with compact support can be extended to a $2$-dimensional foliation on $D^2\times \R^{n-2}$ satisfying \ref{extend1} and \ref{extend2}. 

A smooth periodic homotopy between smooth periodic curves $w_0,w_1:[0,1]\to\DR$ is a smooth map $h: [0,1]\times [0,1]\to \DR$ which is constant on $[0,1]\times\{0\}$ and $[0,1]\times\{1\}$, restricts to $w_i$ on $[0,1]\times \{i\}$, and for which the extension to $\R\times [0,1]$ given by $(t+m,s)\mapsto h(t,s)\circ w_0(1)^m$, $t\in [0,1]$, $m\in \Z$,  is smooth. If two periodic curves are smoothly homotopic then we can find smooth homotopies which are constant near $0$ and $1$ in the sense that the induced curves for small $s$ are equal to the restriction to $[0,1]\times\{0\}$ and for $s$ close to $1$ equal to the restriction to $[0,1]\times\{1\}$. 

\begin{remark}\label{rem:homotIsFol} Notice that a smooth periodic homotopy corresponds to a foliated $\R^{n-2}$-product with compact support over $S^1\times[0,1]$. But be aware that a smooth map $h:\R\times[0,1]\to \DR$ which is periodic in the sense that each restriction to $\R\times\{s\}$, $s\in[0,1]$, is periodic will in general not correspond to a foliated product over $S^1\times[0,1]$.
\end{remark}

Therefore, if $w$ and $w'$ are smoothly periodic homotopic and $w$ is fillable, then so is $w'$. This allows us to look only at curves $w:[0,1]\to \DR$ starting at $id$ which are constant near $0$ and $1$. These are always periodic, and concatenating two of these will give us a new periodic curve.  An interval $J\subset [0,1]$ where a path $w$  is constant we call a horizontal interval (for $w$), since the  foliation on $S^1\times \R^{n-2}$ over the part of $S^1$ corresponding to this interval is horizontal. A smooth path starting at $id$ which is horizontal near its beginning and end we call an {\it adjusted smooth path}.
\begin{lemma}\label{lem:prodfill} The concatenation of two fillable adjusted smooth paths is fillable.
\end{lemma}
\begin{proof} The concatenation of $w_1$ and $w_2$ is the smooth path $w_1\ast w_2$ defined by
 $$
 w_1\ast w_2(t):= \left\{\begin{array}{lcl}
                                     w_1(2t)&:&\textrm{ if } 0\leq t\leq 1/2,\\
                                     w_2(2t-1)\circ w_1(1)&:&\textrm{ if } 1/2\leq t\leq 1.
                                     \end{array}
                            \right.
 $$ 
 Since $w_1$ and $w_2$ are horizontal near $0$ and $1$ there is an $\varepsilon>0$ such that the foliation associated to $w_1\ast w_2$ is horizontal on the two intervals of $S^1$ of points having distance at most $2\varepsilon$ from the real axis. Thus starting with $\sqrt{1-\varepsilon^2}+ i\cdot \varepsilon$ as basepoint and running $w_1$ from there to $-\sqrt{1-\varepsilon^2}+ i\cdot \varepsilon$ at double speed we obtain a smooth path from $id$ to $w_1(1)$. Similarly, starting at $-\sqrt{1-\varepsilon^2}- i\cdot \varepsilon$ and running $w_2\circ w_1(1)$ till $\sqrt{1-\varepsilon^2}- i\cdot \varepsilon$ at double speed we get a smooth path starting at $w_1(1)$ and ending at $w_2(1)\circ w_1(1)$. We extend the first path by concatenating it with the constant path with value $w_1(1)$ on the straight segment from $-\sqrt{1-\varepsilon^2}+ i\cdot \varepsilon$ to $\sqrt{1-\varepsilon^2}+ i\cdot \varepsilon$ and the second one by concatenating the constant path with value $w_1(1)$ on the segment from $\sqrt{1-\varepsilon^2}- i\cdot \varepsilon$ to $-\sqrt{1-\varepsilon^2}- i\cdot \varepsilon$ with it . Up to reparametrization, the first path is $w_1$ while the second one is $w_2\circ w_1(1)$. By hypothesis both are fillable so that we obtain a $2$-dimensional foliation on $(D^2\setminus \{z\in \C:\; |\Im(z)|<\varepsilon\})\times \R^{n-2}$ which is horizontal near the two straight segments with imaginary part equal to $\pm \varepsilon$. We can extend this foliation to all of $D^2\times \R^{n-2}$ by adding the horizontal foliation over $\{z\in \C:\; |\Im(z)|\leq\varepsilon\}$.
\end{proof}
 
 \subsection{A fillable foliated $\R^{n-2}$-product.}\label{ThurExample}
 In this short subsection we present the example of Thurston \cite{Thur}  of a filling for a particular non-trivial smooth  periodic curve. On the $2$-torus $S^1\times S^1$ consider the foliation by lines of constant slope $a$, i.~e. the foliation defined by the closed $1$-form $d\theta-a\,d\varphi$ where $(\varphi,\theta)$ are the coordinates of $S^1\times S^1$. This foliation corresponds to the smooth periodic curve $t\mapsto R_{ at }$ where $R_s\subset \mathrm{Diff}S^1$ is rotation by $2\pi s$. All these foliations allow a filling to a $2$-dimensional foliation of the solid torus $D^2\times S^1$ by first pushing the $1$-dimensional foliation constantly a little to the inside, then turbulizing it in the $\theta$-direction to a foliation which has the torus $\frac{1}{2}S^1\times S^1$ as the limit set of all its leaves, and fill in the remaining solid torus by a Reeb component. 
 
 Thurston proceeds then by looking at the standard tubular neighbourhood $S^1\times D^{n-3}$ of the standard $S^1$ in $\R^{n-2}$. Please {\bf notice} that this is where we have to assume that $n\geq 4$. Then $S^1\times S^1\times D^{n-3}$ is a $D^{n-3}$-family of $2$-tori. Define a smooth periodic curve $[0,1]\to \DR$ which is constant outside of $S^1\times D^{n-3}$ and restricts for each circle $S^1\times \{x\}$, $x\in D^{n-3}$,  to a periodic curve corresponding to a foliation by lines of constant slope $f(x)$ where $f:D^{n-3}\to [0,1]$ is a non-vanishing smooth function which is $0$ in a neighbourhood of the boundary.  Then fill for each $x\in D^{n-3}$ the resulting foliation of $S^1\times S^1\times\{x\}$ to obtain a $2$-dimensional foliation of $D^2\times S^1\times\{x\}$, and do this in such a way that the foliations fit smoothly together. 
 
 There is one further step to take. Outside  $D^2\times C$   the filling  has to be induced by projection to the second factor, for some compact $C\subset \R^{n-2}$. But the filling for the slope $0$ foliations coming from the insertion of Reeb components are not of this type. So one has to interpolate between these two foliations  close to the boundary of $S^1\times D^{n-3}$. This interpolation is achieved by the introduction of the function $g:D^{n-3} \to [0,1]$ in the formula below, where $g$ is supposed to be smooth, to be $0$ in a neighbourhood of the boundary, and equal to $1$ in the support of $f$. 
 
 Again, this interpolation is easy to understand, as will be explained below.
 
 The filling of the constant slope $f(x)$ foliation on $S^1\times S^1$ is given by the following $1$-form on $D^2\times S^1$. Give $D^2$ polar coordinates $(r,\varphi)$, and let $\{\lambda_0,\lambda_{1/2},\lambda_1\}$ be a smooth partition of unity for $[0,1]$ such that $\lambda_i$ is equal to $1$ in a neighbourhood of $i$, and such that $\lambda_0$ and $\lambda_1$ have disjoint support. Then the 1-form
 $$
 \lambda_1(r)(d\theta-f(x)d\varphi) +\lambda_{1/2}(r)dr+\lambda_0(r)d\theta
 $$
 is completely integrable, equal to $(d\theta-f(x)d\varphi)$ near the boundary and thus defines a filling which, in fact, is  as described above, at least qualitatively. To also take care of the interpolation mentioned above  take for each $x\in \R^{n-2}$ the $1$-form 
 $$
 (1-g(x))d\theta + g(x)(\lambda_1(r)(d\theta-f(x)d\varphi) +\lambda_{1/2}(r)dr+\lambda_0(r)d\theta).
 $$
 Again one checks that the form is nowhere zero and integrable. For those $x$, where $g(x)=0$, the form is equal to $d\theta$ on $D^2\times S^1\times\{x\}$ and thus the sets $D^2\times\{(\theta,x)\}$, $\theta \in S^1$, are leaves of the foliation. Therefore, if $x$ is near the boundary of $D^{n-3}$ the foliation is induced by projection onto $\R^{n-2}$, i.~e. it is horizontal there.  This allows us to extend this foliation to all of $D^2\times\R^{n-2}$ by making it horizontal outside $D^2\times S^1\times D^{n-3}$.
 
 To understand the interpolation mentioned above geometrically look at points $x\in D^{n-3}$ where $f(x)=0$. The form there is equal to $$
 (1-g(x)\lambda_{1/2}(r)) d\theta +g(x)\lambda_{1/2}(r) dr.
 $$ 
 Thus the foliation is horizontal outside the support of $g$ and $\lambda_{1/2}$. Inside the support of  $\lambda_{1/2}$ the slope in the $\theta$-direction will increase when moving from small $g$-values to larger ones. In particular, inside the solid torus $\{r\leq \frac{1}{2}\}\times S^1$ you will see a bubble occuring for each leaf when leaving the set $\{g=0\}$ which grows larger and larger with growing values of $g$. The slopes at the $2$-torus $\{r=\frac{1}{2}\}\times S^1$ are getting  steeper an steeper approaching tangentially more and more this torus, and, finally, being this $2$-torus when $g=1$ is reached.

 As the last requirement for an appropriate filling we have to check (see \ref{transversetofactfill}) that the plane field of this foliation is homotopic to a plane field transverse to the $\R^{n-2}$-factor by a homotopy which is constant near $S^1\times \R^{n-2}$. To do this homotope the defining $1$-form by  replacing for $s\in [0,1]$ the function $\lambda_{1/2}$ by $(1-s)\lambda_{1/2}$ and $\lambda_0$ by $\lambda_0+ s\lambda_{1/2}$. The resulting $1$-forms are all nowhere vanishing and the final form is $d\theta - f(x) d\varphi$,  which is transverse to the $\R^{n-2}$-factor.
 
 \subsection{Filling the hole when $\DR_0$ is simply connected.}\label{subsect:easyfill}
 We denote the path component of  $id$ in  $\DR$ by $\DR_0$. From subsection \ref{ThurExample} we have a fillable periodic path $\alpha$ from $id$ to an $f\not= id$ in $\DR_0$. If $w$ is a fillable periodic path then for any $g\in \DR_0$ the path $g\circ w\circ g^{-1}$ is also fillable. The manifold $\DR$ is modelled  on the space $C^\infty_c(\R^{n-2},\R^{n-2})$ of smooth maps from  $\R^{n-2}$ to itself with compact support given the inductive limit topology of $C^\infty_K(\R^{n-2},\R^{n-2})$, where $C^\infty_K(\R^{n-2},\R^{n-2})$ is the space of smooth maps from $\R^{n-2}$ to itself with support in the compact set $K\subset \R^{n-2}$ with its usual $C^\infty$-topology. Thus continuous paths $w:[0,1]\to \DR$ are homotopic to a finite product  of  paths, each of which is a straight path in a coordinate chart. These paths can be made smooth by parameter changes. Also by the usual smooth approximation methods any continuous homotopy between smooth paths gives rise to a smooth homotopy between these paths.
 
 Therefore, if $\DR_0$ is simply connected and simple, then any smooth periodic path can be filled. In fact, for every $g\in \DR_0$ there are $g_1,\ldots,g_r$ with $g= \Pi_{i=1}^k g_ifg_i^{-1}$ and so any smooth periodic curve from $id$ to $g$ is the concatenation of paths from $id$ to $g_ifg_i^{-1}$, $i=1,\ldots,k$, where every one of these paths are fillable by the Reeb type  construction of \ref{ThurExample}. 
 
 In fact, we can reduce the number of these Reeb fillings to $8$ due to results of Takashi Tsuboi in \cite{Tsu:unifsimp}. Theorem 2.1 of this paper says that any element of $\DR_0$ is the product of two commutators. Using this and Lemma 3.1 of \cite{Tsu:unifsimp} every element of $\DR_0$ can be expressed as a product of eight conjugates of $f$ and $f^{-1}$. 
 
 Since the Reeb fillings of the periodic curve to $f$  of subsection \ref{ThurExample} is very explicit, this gives us an explicit description of how the holes are filled.
 
 \begin{remark}\label{vanishhgroups}
 For $n=4$ and $n=5$ the spaces $\DR_0$ are simply connected so that the above  simple filling of the holes can be done. In fact,  Smale \cite{smale} (see also section 5 of the Appendix of \cite{cerf}) for $k=2$ and  Hatcher \cite{hatch} for $k=3$ show that all homotopy groups of the group $\mathrm{Diff}(D^k,S^{k-1})$ of diffeomorphisms of $D^k$ which are the identity on $S^{k-1}$ are trivial. As a consequence, for $n$ equal to $4$ or $5$ all homotopy groups of $\DR_0$ are trivial. To see this note that $\pi_i\DRk_0\cong \pi_i\mathrm{Diff}_c\mathring{D}^k_0$  is the inductive limit of $\pi_i\mathrm{Diff}_{D_j}\mathring{D}^k$ where $D_j$ is the disk of radius $1-1/j$ in $\mathring{D}^k$. But $\pi_i\mathrm{Diff}_{D_j}\mathring{D}^k =\pi_i \mathrm{Diff}(D_j, J^\infty\partial D_j)$. Here $ \mathrm{Diff}(D_j, J^\infty\partial D_j)$ is the group of diffeomorphisms of $D_j$ which are the identity on $\partial D_j$ and are infinitely tangent to the identity of $D_j$ there. By the special case of Proposition 1 in the Appendix of \cite{cerf} the inclusion $\mathrm{Diff}(D^k,J^r\partial D^k)\to \mathrm{Diff}(D^k,\partial D^k)$ induces for all $0\leq r\leq\infty$  isomorphisms for all homotopy groups, and we are done.
 \end{remark}

\section{Filling the hole. Part II}\label{sect:hole2}
In this section we show how the holes can be filled without assuming that $\DR_0$ is simply connected. The idea is fairly simple. We break the periodic path  in $\DR$ at the boundary of the hole into short pieces of the form $w_i\circ g_i$ where each $w_i$ is a path inside a predescribed contractible neighborhood of the identity and $g_i$ some element of $\DR_0$. Then we show using results from \cite{HalRyb} that each $w_i(1)$ is the product of $6$ commutators of elements in $\DR$ still lying in some prescribed contractible neighbourhood of the identity and then, using not only Lemma 3.1 of \cite{Tsu:unifsimp} but also its proof, we show that each $w_i$  is homotopic to the product of at most $24$ conjugates of the Thurston example from subsection \ref{ThurExample}.Then by Lemma \ref{lem:prodfill} we are done. 

There is one issue that will attribute a little to the length of this section. The authors of \cite{HalRyb} use for the spaces of smooth mappings between smooth manifolds the  topology introduced in Chapter~42 of \cite{KriegMich} which differs for the case of non-compact domain manifolds from the one that most people are used to. Chapter~42 of \cite{KriegMich} uses quite a few notions and results from earlier Chapters which would take  some effort to absorb. But in fact we will need very little from \cite{KriegMich} to understand the arguments of this section. Naturally, one has to invest more if one wants to understand the proof of Proposition 2 in \cite{HalRyb}, the result from \cite{HalRyb} which we need in our proof. 
\subsection{Smooth perfectness of $\DRk_0$ (following Haller, Rybicki and Teichmann).}
\label{subsec:smoothperf} The notion of a differentiable map $f$ from $\R$ into a locally convex vector space $E$ is the obvious one: difference quotients at each point  $x\in\R$ converge in $E$ to a point called $f'(x)$. Iterating this gives the notion of a smooth curve. Similarly one can define smooth maps from  open subsets of $\R^m$ into $E$ by requesting that all partial derivatives exist and are locally bounded. The locally compact vector space of interest to us is the space $C^\infty_c(\R^k,\R^k)$ of smooth maps from $\R^k$ to $\R^k$ with compact support. It is given  the inductive limit topology of the subspaces $C^\infty_K(\R^k,\R^k)$ of smooth maps with support in the compact set $K\subset \R^k$ with the usual $C^\infty$-topology. Any compact subset of $\CRk$ is contained in some $C^\infty_K(\R^k,\R^k)$.  Therefore, for any continuous, and in particular, any smooth map $f: \R\to C^\infty_c(\R^k,\R^k)$ and any bounded $J\subset \R$ there is a compact $K\subset \R^k$ such that $f(J)\subset C^\infty_K(\R^k,\R^k)$. 

The space $\DRk$ is a manifold modelled on $\CRk$. As an atlas take $\{h_g:g+U_g\to U_g\vert g\in \DRk\}$. Here $U_g$ is an open neighbourhood of $0$ in $\CRk$ such that $g+U_g\subset \DRk$ and $h_g(g+f)=f$ for $f\in U_g$. Coordinate changes are then simply restrictions of translations in $\CRk$ to open sets. Clearly, a map $w$ from an open subset of $\R^m$ into $\CRk$ is smooth if and only if  $w+g$ is smooth for any $g\in \CRk$. Thus it makes sense to call a map from a finite dimensional smooth manifold into $\DRk$ smooth if its restriction to inverse images of the above charts are smooth as maps into $\CRk$. 

If we want to consider smoothness of maps from $\DRk$ into some other possibly infinite dimensional manifold modelled on a locally convex vector space we need to decide which maps from an open subset of a  locally convex vector space into another are considered to be smooth. The concept should be so that on finite dimensional ones exactly the usual smooth maps  are smooth.

In \cite{KriegMich} a map between locally convex vector spaces is called smooth if it maps smooth curves to smooth curves. By Bomann's theorem \cite{bomann} (see also Corollary 3.14 in \cite{KriegMich}) a map from an open subset of $\R^m$ into any locally convex space is smooth in this sense if and only if it is smooth in the usual sense. 

We want to point out that it is essential here that we are in the smooth category. In the $C^r$-category with  $r$ finite the corresponding concept will not be sufficient to assure that a map on an open subset of $\R^m$ into $\R$ is smooth in the usual sense (see e.~g. \cite{KriegMich}, Example 3.3; see also the subsection {\it ``Smoothness of foliated products''} in \S 4 of \cite{Tsu:conn}).

The infinite dimensional manifolds for which smoothness structures are developed in \cite{KriegMich} are modelled on so called convenient vector spaces. These are locally convex vector spaces $E$ for which a map $w:\R\to E$ is smooth if and only if for all continuous linear $l:E\to\R$ the map  $l\circ w:\R\to\R$ is smooth. The space $\CRk$ is convenient. This follows for example from Theorem 2.15 in \cite{KriegMich} using the fact that any Banach space is convenient. The authors of \cite{KriegMich} then go one step further by using the $c^\infty$-topology on convenient vector spaces. This, by definition (\cite{KriegMich}, Definition 2.12) is the final topology with respect to all smooth curves. For Fr\'echet spaces the $c^\infty$-topology agrees with the original topology, in particular this holds for $C^\infty_K(\R^k,\R^k)$. But in general, the $c^\infty$-topology is strictly finer. By  Proposition 4.26~(ii) of \cite{KriegMich} the space $\CRk$ with the $c^\infty$-topology is not a topological vector space. Thus the $c^\infty$-topology differs from the original topology.

The $c^\infty$-topology is used in the results of \cite{HalRyb} which we use. But this will not cause any additional difficulty for us. We define a map $f$ defined on a $c^\infty$-open subset $U$ of a locally convex topological vector space into a locally convex vector space $F$ to be smooth if it maps smooth curves in $U$ to smooth curves in $F$. 

It is worthwhile to point out the obvious fact that a smooth map from  a $c^\infty$-open subset $U$ of a locally convex topological vector space into a locally convex vector space $F$ is continuous with respect to the $c^\infty$-topology used for both, $U$ and $F$. In particular, inverse images of open sets of $F$ in its usual topology are open in $U$ with the $c^\infty$-topology.

 In \cite{KriegMich} charts of manifolds modelled on convenient vector spaces are bijections onto $c^\infty$-open subsets. So  $g\in \DRk$ may now have smaller neighbourhoods. In our atlas above we may replace $U_g$ by $c^\infty$-open neighbourhoods of $0$ in $\CRk$.

We can now state the result of \cite{HalRyb} that we are going to use.
\begin{proposition}[Proposition 2 of \cite{HalRyb}]\label{prop:HalRyb} Suppose $k\geq2$, and let $B\subset \R^k$ be open and bounded. Then there exist compactly supported smooth vector fields $X_1,\ldots, X_6$ on $\R^k$, a $c^\infty$-open neighborhood $\cW$ of the identity in $\mathrm{Diff}^\infty_c(B)$, and smooth mappings $\sigma_1,\ldots,\sigma_6:\cW\to \DRk$ so that $\sigma_i(id)=id $ and for all $g\in \cW$,
$$ g=[\sigma_1(g),\textrm{exp}X_1]\circ\cdots\circ[\sigma_6(g),\textrm{exp}X_6].
$$
Moreover the vector fields $X_i$ may be chosen arbitrarily close to $0$ with regard to the strong Whitney $C^\infty$-topology.
\end{proposition}
 
As usual $[a,b]=a\circ b\circ a^{-1}\circ b^{-1}$, and ${exp}(X)$ is the time $1$ map of the flow associated to the complete vector field $X$.

\subsection{From uniform perfectness to uniform simplicity (following Tsuboi).}\label{subsec:perfecttosimple} In this subsection we copy the proof of Lemma 3.1 in \cite{Tsu:unifsimp} practically verbatim for later reference.
\begin{remark*} Takashi Tsuboi pointed out that Lemma 3.1 in \cite{Tsu:unifsimp} can easily be derived from a paper of R. D. Anderson (see the proof of  Lemma~1 in \cite{Anderson}). But what we need we find word for word  in  \cite{Tsu:unifsimp}:  Tsuboi has already distilled exactly what we need from the paper of Anderson. 
\end{remark*}
\begin{lemma}\label{lem:simptrick} Let $U$ be an open subset of the manifold $M$ and let $h$ be a diffeomorphism of $M$ such that $U\cap h(U)=\emptyset$. Let $a,b\in \mathrm{Diff}^\infty(M)$ have support in $U$. Then $[a,b]$ is a product of four conjugates of $h$ and $h^{-1}$. 
\end{lemma}
\begin{proof} Consider $c=h^{-1}ah$. Then the support of $c$ and $b$ are disjoint, and consequently they commute. Therefore
$$ 
\begin{array}{rcl}
aba^{-1}b^{-1} & = & h(h^{-1}ah)h^{-1} (ba^{-1}b^{-1})\\
& = & h(ch^{-1}c^{-1})(cbhc^{-1}b^{-1})(bh^{-1}b^{-1})\\
& = & h(ch^{-1}c^{-1})(bchc^{-1}b^{-1})(bh^{-1}b^{-1}).
\end{array}
$$
\end{proof}
\begin{corollary}\label{cor:perftosimp} Let $M^k$ be a smooth connected manifold and let $h\in \mathrm{Diff}^\infty(M^k)$ be any element different from the identity. Let $a_i,b_i$, $i=1,\ldots,r$, be elements of $\mathrm{Diff}^\infty(M^k)$ such that for each $i$ the diffeomorphisms  $a_i$ and $b_i$ have support in $\mathring{U}_i$ where $U_i$ is a  closed $k$-ball in $M^k$.

Then $f:=\prod\limits^r_{i=1}[a_i,b_i]$ is a product of $4r$ conjugates of $h$ and $h^{-1}$.
\end{corollary}
\begin{proof} Since $h\neq id$ there exists a closed $k$-ball $U$ in $M$ with $h(U)\cap U=\emptyset$. Since $M$ is connected there is for each $i$ a diffeomorphism $g_i$ of $M$ smoothly isotopic to $id$ by a compactly supported isotopy, such that $g_i(U_i)\subset U$. Then by Lemma \ref{lem:simptrick} $g_i[a_i,b_i]g_i^{-1}$ is a product of four conjugates of $h$ and $h^{-1}$, and thus $[a_i,b_i]$ is a product of four conjugates of $h$ and $h^{-1}$.
\end{proof}

\subsection{Putting it all together.}\label{finalstep} We finally show that every smooth periodic curve $w:[0,1]\to\DR$ can be filled. 

As explained in the introduction of this section, the idea is to write $w$ as the product of periodic paths $w_i$ where each $w_i$ is contained in a contractible neighbourhood of $id$ in $\DR$. We then proceed as in Subsection \ref{subsect:easyfill} by working entirely with paths which are homotopic to paths contained in this contractible neighbourhood. That this can be done is a consequence of the results \ref{prop:HalRyb}, \ref{lem:simptrick}, and \ref{cor:perftosimp} of the previous two subsections. 

To shorten notation a little let $k:=n-2$. Thus $k\geq 2$. 

\begin{itemize}
\item[(1)] Let $V:=\{id + e\vert e\in \CRk \textrm{ with } \max\limits_x \|de_x\| <1\}$. This is an open contractible neighbourhood of $id$ in $\DRk$. An easy estimate shows that the composition  of $r$ elements of $\{id + e\vert e\in \CRk \textrm{ with } \max\limits_x \|de_x\| <\varepsilon\}$ is of the form $id + f$ with $\|df_x\|<(1+\varepsilon)^r-1$ for all $x$. Therefore there is $\varepsilon>0$ such that any composition of $72$ elements of $V_0:=\{id + e\vert e\in \CRk \textrm{ with } \max\limits_x \|de_x\| <\varepsilon\}$ is contained in $V$.
\item[(2)] Recall from Subsection \ref{ThurExample} that the periodic curves which get filled are of  the form
$$
h_f(t)(z)=\left\{\begin{array}{lcl}
z&,&z\not\in S^1\times D^{k-1}\\
(\theta+t\cdot f(x),x)&,&z=(\theta,x)\in S^1\times D^{k-1}
\end{array}
\right.
$$
where $f: D^{k-1}\to [0,1]$ is smooth and vanishing near the boundary. We have $h_f(t)^{-1}=h_{-f}(t)$ for all $t\in [0,1]$. Thus by choosing $f$ small enough, we may assume that $h_f(t), h_f(t)^{-1} \in V_0$ for all $t\in[0,1]$. With this choice of $f$ we denote $h_f(1)$ by $h$.
\item[(3)] Let $U\subset S^1\times D^{k-1}$ be an open ball such that $U\cap h(U)=\emptyset$. Let $B\subset \R^k$ be an open ball containing $S^1\times D^{k-1}$ and $A$ an open ball in $\R^k$ containing $\bar{B}$. 
\item[(4)] Let $g_t$ be a compactly supported isotopy of $\R^k$ with $g_0= id$ and $g_1(\bar{A})\subset U$. Denote $g_1$ by $g$.
\item[(5)] Finally let $\cW\subset V$ be a $c^\infty$-open neighbourhood of $id$  in $\mathrm{Diff}^\infty_c(B)$ and $\sigma_1,\ldots,\sigma_6: \cW\to \DR$  and $X_1,\ldots, X_6$ the smooth maps and compactly supported vectorfields of Proposition \ref{prop:HalRyb}. Since $A$ is diffeomorphic to $\R^k$ we may assume that the $\sigma_i$ factor through $\mathrm{Diff}^\infty_c(A)$ and that the $X_i$ have support in $A$. Furthermore, by making $\cW$ small enough, we may assume that for $a\in\cW$ the elements $h^{-1}\circ g\circ \sigma_i(a)\circ g^{-1}\circ h$ and their inverses are all in $V_0$. Also, by having $X_i$ close enough to $0$ in the strong Whitney $C^\infty$-topology, we may assume that all $g\circ exp(X_i)\circ g^{-1}$ and their inverses are in $V_0$.
\end{itemize}
Now let $w:[0,1]\to \DRk$ be a smooth periodic curve. The curve $w$ is fillable if (and only if) the curve $g\circ w\circ g^{-1}$ is fillable. Here $g$ is the diffeomorphism from item (4) above.

By the very definition of the $c^\infty$-topology $g\circ w\circ g^{-1}$ is continuous with respect to this topology on $\DRk$. Therefore, we find a $q\in\N$ such that for $i=0,\ldots,q-1$ the curves $gw_ig^{-1}:[0,1]\to \DRk$ defined by
$
w_i(t):= w(\frac{i+t}{q})\circ w(\frac{i}{q})^{-1}
$
have their images in $\cW$. Up to parametrization $gwg^{-1}$ is the concatenation of the $gw_ig^{-1}$. By making $w$ horizontal near the points $i/q$ the $gw_ig^{-1}$ are periodic.

 By the proof of Corollary \ref{cor:perftosimp} each commutator of two diffeomorphisms $g\alpha g^{-1}$ and $g\beta g^{-1}$ is a product of four conjugates of $h$ and $h^{-1}$, if the support of $\alpha$ and $\beta$ is in $A$, and by the proof of Lemma \ref{lem:simptrick} we have an explicit description of these factors. 

For the case of interest to us it works out as follows. Let $a\in \cW$ and $i=1,\ldots,6$. Then
$g[\sigma_i(a), exp X_i]g^{-1}
$
equals
$$
\begin{array}{l}
 h\circ ((h^{-1}g\sigma_i(a) g^{-1}h)\circ h^{-1}\circ(h^{-1}g\sigma_i(a)^{-1} g^{-1}h))\\
\circ ((g\, exp X_i g^{-1})\circ (h^{-1}g\sigma_i(a) g^{-1}h)\circ h\circ (h^{-1}g\sigma_i(a)^{-1} g^{-1}h)\circ(g\, exp X_i^{-1} g^{-1}))\\
\circ((g\, exp X_i g^{-1})\circ h^{-1}\circ(g\, exp X_i ^{-1}g^{-1})).
\end{array}
$$
This is the composition of 12 diffeomorphisms with each one lying in $V_0$ by item (5) above if $a\in \cW$. Since there are six of these commutators to express $a$ we look at a composition of $72$ diffeomorphisms of $V_0$. By item (1) this composition is contained in $V$. The same holds if we replace the first $h$ and the occurrences of the $h$ or $h^{-1}$ in the center of the outer parantheses of the other three factors by $h_f(t)$ respectively $h_f(t)^{-1}$. This follows from item (2) above.

Having done this, we obtain a smooth periodic curve $u$  in $V$ running from $ id $ to $gw_i(1)g^{-1}$. Thus $u$ and $gw_ig^{-1}$ are homotopic. 

The curve $u$ is the product of $24$ fillable curves with the product being composition in $\DRk$ instead of concatenation. But up to homotopy there is no difference. In fact, if each of these curves is horizontal near $0$ and $1$ the standard homotopy between them is smoothly periodic.

We are done.

\section{Application: $\pi_{k+1}B\overline{\Gamma}_k=0$}\label{sect:applE}
As mentioned in the introduction the proof that any $2$-plane field is homotopic to a foliation follows immediately from the main result, Theorem 1 in \cite{Thur} once you know that the homotopy fibre $B\bar{\Gamma}_k$ of the normal bundle map $\nu:B\Gamma_k\to BGL_k$, which assigns to a smooth $\Gamma_k$-structure its normal bundle, is $(k+1)$-connected. Mildly reworded Theorem 1 of \cite{Thur} can be formulated as:
\begin{theorem}[Thurston]\label{thm:Thurhprinc} Let $M$ be a smooth manifold and $TM=\tau\oplus\rho$ a splitting of its tangent bundle such that $\rho$ is the normal bundle of a Haefliger structure $\cH$ of codimension $k\geq 2$ on $M$. Then there exists a smooth foliation $\F$ of codimension $k$ on $M$ which is homotopic to $\cH$ as a Haefliger structure and whose tangent bundle is homotopic to $\tau$ as a subbundle of $TM$. Furthermore, if $\cH$ is already the Haefliger structure  of a foliation in a neighbourhood of a compact $K\subset M$ then the homotopies between $\F$ and $\cH$ and between the tangent field of $\F$ and $\tau$ can be chosen to be constant in a neighbourhood of $K$.
\end{theorem}
This theorem is sometimes called Thurston's $h$-principle for foliations of codimension greater than $1$.

Thus,  a given $p$-plane field on the $n$-manifold $M$ is homotopic to a foliation if its normal bundle $\rho$ comes from a $\Gamma_{n-p}$-structure, i. e. if the classifying map $f_\rho: M\to BGL_{n-p}$ of $\rho$ can be liftet to $B\Gamma_{n-p}$. This can always be done if $B\bar{\Gamma}_{n-p}$ is $(n-1)$-connected. 

Thus any $2$-plane field on any manifold can be homotoped into a foliation if $B\bar{\Gamma}_k$ is $(k+1)$-connected, and any $p$-plane field on any manifold is homotopic to a foliation if $B\bar{\Gamma}_k$ is $(k+p-1)$-connected. It is known that $\pi_{2k+1}(B\bar{\Gamma}_k)$ is highly non trivial, and it is a long standing open question whether $B\bar{\Gamma}_k$ is $2k$-connected.

Haefliger in \cite{Hae:ouvert} used the Gromov-Phillips Transversality Theorem (see below and \cite{Grom}, \cite{Phil}) to show that $B\bar{\Gamma}_k$ is $k$-connected. The proof that $\pi_{k+1}B\bar{\Gamma}_k=0$  is more involved. It uses the Thurston-Mather Theorem which states that the adjoint $B\overline{\mathrm{Diff}}^\infty_c(\R^k)\to \Omega^kB\bar{\Gamma}_k$ of the natural map $\Sigma^nB\overline{\mathrm{Diff}}^\infty_c(\R^k)\to B\bar{\Gamma}_k$ is a homology equivalence, and it uses Thurston's theorem \cite{ThurBull} that the universal covering group of $\DRk_0$ is perfect. The last theorem is also used by us to prove our main result (see Subsection \ref{subsec:smoothperf}).

We close this article by showing that the vanishing of $\pi_{k+1}B\bar{\Gamma}_k$  for $k\geq2$ is a direct consequence of the Gromov-Phillips Transversality Theorem and our main theorem \ref{thm:main}. Our proof follows closely Haefliger's  proof in \cite{Hae:ouvert} that $ \pi_{k}B\bar{\Gamma}_k=0$. In fact our argument generalizes immediately to prove the following 

\begin{theorem}\label{intfieldstovanishing} Let $p\geq 0$ and assume that any  $p$-plane field on any smooth manifold $M$, which in a neighbourhood of the compact subset $K$ of $M$ is the tangent plane field of a smooth foliation, is homotopic to the plane field of a smooth foliation by a homotopy which is constant in a neighbourhood of $K$.

Then $\pi_{k+p-1}B\bar{\Gamma}_k=0$

\end{theorem}
Here, as before, $\Gamma_k$ is the space of germs of local $C^\infty$-diffeomorphisms of $\R^k$ with the germ topology.

The Gromov-Phillips Theorem as stated in \cite{Phil}, a special case of the main result in \cite{Grom}, suffices for our needs. It says the following:
\begin{theorem}[Gromov, Phillips] \label{thm:grromphil} Let $M$ and $N$ be smooth manifolds and $\cG$ a foliation of $N$. Let $\mathrm{Trans}(M;(N,\cG))$ be the space of smooth maps from $M$ to $N$ which are transverse to $\cG$ with the $C^1$ compact open topology, and let $\mathrm{Epi}(TM, \nu_\cG)$ be the space of continuous bundle epimorphisms from the tangent bundle $TM$ of $M$ to the normal bundle $\nu_\cG$ of $\cG$ with the compact open topology. 

Then the natural map which maps $g\in \mathrm{Trans}(M;(N,\cG))$ to $\mathrm{pr}\circ dg\in \mathrm{Epi}(TM, \nu_\cG)$ is a weak homotopy equivalence if $M$ is open (i. e. if $M\setminus\partial M$ has no compact component).

Here $\mathrm{pr}:TN\to \nu_\cG$ is the obvious projection.
\end{theorem} 

\begin{proof}[Proof of \textup{Theorem \ref{intfieldstovanishing}}] A map $g:S^{k+p-1}\to B\bar{\Gamma}_k$ corresponds to a Haefliger structure $\cH$ of codimension $k$ on $S^{k+p-1}$ together with a homotopy class of trivializations of its normal bundle. Actually when using the Milnor model for classifying spaces of topological groupoids (see  Section I.5 of \cite{Hae:integrability}) we have a definite trivialization (see e. g. \cite{Hus}, Chapter 4, Section 9). Changing the Haefliger structure in its homotopy class induces homotopies of the  normal bundle and thus a homotopy of $g$. 

We are only interested in homotopy information. Thus, see for example section 1.9 of \cite{Hae:ouvert}, we may think of $\cH$ as the restriction to the zero section of a foliation $\cG$ 
on the total space $E$ of a smooth rank $k$ vectorbundle $p:E\to S^{k+p-1}$ with $\cG$ transverse to the fibres of $p$. Furthermore, the normal bundle of $\cH$  is the restriction of the normal bundle of the foliation $\cG$ to the zero section. The normal bundle $\nu_\cG$  of $\cG$ is the tangent field to the fibres of $p$, so that the normal bundle of $\cH$ is $p$. Any trivialization of $p$ will extend to a trivialization of $\nu_\cG$. 

Let $M:=S^{k+p-1}\times (0,2)$ be the subset of $\R^{k+p}$ which is the union of $(k+p-1)$ spheres of radius $r\in (0,2)$ with center the origin. Then $M$ is an open manifold, $(E,\cG)$ is a foliated manifold so that the Gromov-Phillips Theorem applies. 

Let $f_0:M\to E$ be given by $f_0(x,t)=x$, $x\in S^{k+p-1}$, $t\in(0,2)$, where we identify $S^{k+p-1}$ with the zero section of $p$, and let $\varphi: E\to S^{k+p-1}\times \R^k$ be the trivialization of the normal bundle given by  our map $g$ into $B\bar{\Gamma}_k$. 

Let $TM=M\times\R^{k+p}$ be the standard framing of the open subset $M$ of $\R^{k+p}$, and consider the bundle epimorphism $\tilde{f}_0:TM\to \nu_\cG$, which maps $(x,v)\in M\times\R^{k+p}$ to $\varphi^{-1}(f_0(x), p_k(v))$, where $p_k:\R^{k+p}\to\R^k$ is the projection onto the first $k$ coordinates.

The Gromov-Phillips Theorem then supplies us with a homotopy $\tilde{f}_t:TM\to \nu_\cG$ of bundle epimorphisms over a homotopy $f_t:M\to E$  starting with $\tilde{f}_0$ over $f_0$ and ending with  $\tilde{f}_1= pr\circ df_1$ where $f_1$ is a smooth map transverse to $\cG$ and $pr: TE\to \nu_\cG$ is the projection map.
 
 Restricting the homotopy $f_t$ to $S^{k+p-1}$ gives us a homotopy of Haefliger structures on $S^{k+p-1}$ from $\cH=f^*_0\cG$ to  $f_1^*\cG$ restricted to $S^{k+p-1}$. The normal bundles of these Haefliger structures are the pull backs of $\nu_\cG$ via $f_t$. Together with the trivialization of $\nu_\cG$ we see that up to homotopy we can assume that the initial Haefliger structure is the one induced from $\cG$ by $f_1$ restricted to $S^{k+p-1}$. But this Haefliger structure is the restriction of an honest foliation $\F$ of $M$ to the unit sphere. So restricting $\F$ to $S^{k+p-1}\times\{1/2\}$ gives us a Haefliger structure over $S^{k+p-1}$ (identified with $S^{k+p-1}\times \{1/2\}$) with a trivialization homotopic as a pair to the original pair. We now restrict the homotopy $\tilde{f}_t$ to $S^{k+p-1}\times\{1/2\}$ going backwards from $t=1$ to $t=0$. The kernel of each of these restrictions of bundle maps $\tilde{f}_t$  gives us a continuous family $\tau_t$ of continuous fields of $p$-planes of $\R^{k+p}$, starting with the restriction $\tau_1$ of the tangent bundle of $\F$ to  $S^{k+p-1}\times\{1/2\}$ and ending with the constant field $\tau_0$ of planes orthogonal to $\R^k=p_k(\R^{k+p})$.  Spreading these fields out over $S^{k+p-1}\times [1/4,1/2]$ by putting $\tau_t$ on $S^{k+p-1}\times \{(1+t)/4\}$ gives us a $p$-plane field that can be extended to all of $D^{k+p}$ by using the tangent field of $\F$ for points with distance at least $1/2$ from the origin and by the constant field orthogonal to $\R^k$ for points in the disk of radius $1/4$. This field can be approximated in its homotopy class by a smooth field $\tau$ which is tangent to $\F$  outside a small neighbourhood of the disk of radius $1/2$, and the constant field orthogonal to $\R^k$ in a neighbourhood of the origin. Notice that the trivialization of the normal bundle of this plane field is on each sphere $S^{k+p-1}\times\{a\}$ with $a$ close to $0$, up to rescaling the radius of the sphere, equal to $\varphi\circ \varphi^{-1}$, i. e. the identity of $S^{k+p-1}\times \R^k$, which extends trivially to the normal bundle of $\tau$ over $D^{k+p}$.
 
 The hypothesis of Theorem \ref{intfieldstovanishing} then gives us a homotopy of $\tau$ to a smooth foliation by a homotopy which is constant in a neighbourhood of the boundary $S^{k+p-1}$ and in a neighbourhood of the origin. Since a foliation is a Haefliger structure and homotopies of plane fields give homotopies of their normal fields, we are done.
\end{proof}

Essentially the same proof
gives the following
\begin{theorem}[Haefliger \cite{Hae:integrability}]\label{thm:connlessk} $B\bar{\Gamma}_k$ is $(k-1)$-connected.
\end{theorem}
 \begin{proof} Here we start with a Haefliger structure on $S^{i-1}$ with $i\leq k$. Take now for $M$ the open subset $(S^{i-1}\times (0,2))\times\R^{k-i}$ in $\R^k$. Then the initial bundle epimorphism $TM\to \nu_\cG$ is given by $((x,t,y),v)\mapsto \varphi^{-1}(x,v)$, $x\in S^{i-1}$, $t\in (0,2)$, $y\in\R^{k-i}$, $v\in \R^k$. Then everything proceeds as before and with $p=0$. For $p=0$ the hypothesis of \ref{intfieldstovanishing} is trivially satisfied.
 \end{proof}
 
\begin{remark}
Theorem \ref{thm:main} is not used in the proof of Theorem \ref{thm:connlessk} and in the verification of the hypothesis  of Theorem \ref{intfieldstovanishing} for $p$ equal to $0$ and $1$. Thus $B\bar{\Gamma}_k$ is $k$-connected for any $k$. Theorem \ref{thm:main} shows that for $p=2$ the hypothesis of Theorem \ref{intfieldstovanishing} holds if $k\geq2$. In fact, for $k=1$, even in the very special situation that we would encounter in a  proof along the lines described above the hypothesis is not satisfied: not every smooth $2$-plane field on the open disk of radius $2$ in $\R^3$ which is a foliation in a neighbourhood of $S^2$ is homotopic to a smooth foliation by a homotopy which is constant in a neighbourhood of $S^2$. The proof that $\pi_2B\bar{\Gamma}_1=0$ is in \cite{Ma:integr} based on the Thurston-Mather Theorem for codimension 1 foliations, the proof of which is the main result of \cite{Ma:integr}.
\end{remark}

\vspace*{5pt}

\noindent
Yoshihiko MITSUMATSU\\
{\it Department of Mathematics,  
Chuo University\\
1-13-27 Kasuga, Bunkyo-ku, 
Tokyo, 113-8551, Japan
\\
E-mail address: }{\tt yoshi@math.chuo-u.ac.jp}
\vspace{10pt}
\\
Elmar VOGT\\
{\it Mathematisches Institut, 
Freie Universit\"at Berlin\\
Arnimallee 7, 14195 Berlin, Germany
\\
E-mail address: }{\tt evogt@zedat.fu-berlin.de}

\end{document}